\documentclass[journal,onecolumn]{IEEEtran}
\ifCLASSINFOpdf
\else
\fi
\usepackage{amsmath,amsfonts,amsthm,amssymb,mathrsfs}
\newtheorem{theorem}{Theorem}
\newtheorem{definition}{Definition}
\newtheorem{assumption}{Assumption}
\newtheorem{proposition}{Proposition}

\newtheorem{lemma}{Lemma}

\renewcommand{\gg}{\mathfrak{g}}

\newcommand{\FF}{\mathscr{F}}

\newcommand{\real}{\mathbb{R}}
\newcommand{\MC}{\mathbb{C}}

\renewcommand{\hat}{\widehat}

\newcommand{\isp}{\langle \cdot, \cdot \rangle}

\newcommand{\eps}{\varepsilon}

\newcommand{\Id}{\textup{Id}}

\newcommand{\ME}{\mathbb{E}}

\newcommand{\unitary}{\textbf{\textup{U}}}

\newcommand{\SO}{\textbf{\textup{SO}}}
\newcommand{\hh}{\mathfrak{h}}
\newcommand{\pp}{\mathfrak{p}}
\newcommand{\MP}{\mathbb{P}}
\renewcommand{\Re}{\mathfrak{Re}}
\newcommand{\MV}{\mathbb{V}}

\newcommand{\Log}{\textup{Log}}
\DeclareMathOperator*{\argmin}{arg\,min}

\begin{document}
\title{Decompounding on Compact Symmetric Spaces%\\ for IEEE Transactions on Information Theory
}

\author{Erik~Kennerland% <-this % stops a space
\thanks{This publication is partially financed by the
Dutch Research Council (NWO) under the project number OCENW.M.24.104.}% <-this % stops a space
\thanks{Erik Kennerland is affiliated with Delft University of Technology.}% <-this % stops a space
%\thanks{Manuscript received Month Day, Year; revised Month Day, Year.}
}

\maketitle

\begin{abstract}
This paper examines a stochastic deconvolution problem on compact symmetric spaces which is referred to as decompounding. This involves estimating the step distributions of a random walk, where in addition the number of steps between observations is unknown. The harmonic analysis of symmetric spaces is used to construct an estimator to the problem which converges in mean squared error, extending and improving on the analogous problem on compact Lie groups. The rates of convergence are shown to coincide with asymptotic lower bounds of density estimation in Euclidean space. We provide proofs that while the same rates hold for general density estimation problems in compact symmetric spaces, the decompounding problem lies in a subclass of these with different lower bounds depending on the rank of the space. Consequently, the optimality of the estimator depends on the rank of the symmetric space. Decompounding is a broad problem which appears in applications ranging from mathematical finance to wave optics, and the extension to compact symmetric spaces covers manifolds that commonly appear in the statistics literature.

\end{abstract}


% \fi
\IEEEpeerreviewmaketitle

\section{Introduction}
\IEEEPARstart{M}{}any signal processing and statistical inference problems involve having observations of a latent process that have been perturbed or transformed by an accumulation of multiple actions. We can model this situation abstractly by assuming that the possible transformations form a group which acts on either itself or on some manifold, and suppose that the number of actions on the object itself is random. The recovery of the distribution of the original signal from such observations is referred to as decompounding and is a problem that has received attention in both statistical inference and pure probability theory. The problem was first termed in \cite{Buchmann} in the context of insurance mathematics as a prediction problem of the occurrences of random fluctuations in the value of an asset. In this setting, it is realized for real-valued signals and corresponds to the study of random sums 
$$\sum_{k \leq T} X_k$$
of i.i.d.\ variables $X_k$, with the crucial distinction that the number of terms to include is randomly determined by $T$. The abstraction of the problem to the compact Lie group setting followed in the work of \cite{Said} which was in part motivated by physical applications in modeling of wave scattering on a sphere, but it also laid the foundations for decompounding estimates in the non-Euclidean setting. The use of non-commutative harmonic analysis of Lie groups is crucial for this extension. It is also a strong toolkit for general statistical problems on Lie groups, which have been studied quite extensively for general linear deconvolution problems \cite{Yazici,KooKim} especially with regards to the classical matrix Lie groups \cite{KimSO}. An extension to symmetric spaces is natural from a theoretical standpoint since we can apply techniques from the harmonic analysis of compact Riemannian manifolds. It also better captures the idea of decompounding as the accumulation of random transformations of an object, since in the symmetric space setting we can relax the assumption of the object itself constituting a group. Instead, only an action by a group of transformations is necessary.

Compact symmetric spaces are often considered in more general statistical problems too. The fact that they are classified up to isomorphism \cite{Helgason}, due in large part to the seminal work of Élie Cartan, is a motivation in its own since canonical choices of bases of eigenfunctions are explicitly known and computable for many of the compact symmetric spaces. Among them are the unit spheres $S^d \subset \real^{d+1}$ which lay the foundations of the field of directional statistics, a field with applications ranging from astronomy to geological sciences \cite{Fisher}. Recent interest in statistical modeling of $S^d$ has also arisen in the training of large language models \cite{Shalova}, as well as in computer vision and image recognition \cite{Liu} since, roughly speaking, the intrinsic geometry of the sphere allows for small changes in a dataset to have large contribution in weight as opposed to in modeling on $\real^d$. Extending the theory of decompounding to these cases is non-trivial since, except in the exceptional cases $d = 0,1,3$, the sphere $S^d$ fails to admit a Lie group structure. Other compact symmetric spaces of relevance in the statistics literature are the Grassmannians $\textup{Gr}(k,n) \cong \SO(n)/(\SO(k)\times \SO(n-k))$ for $0 \leq k < n$, which are typically used for models in image recognition for their ability to approximate features of images on lower-dimensional spaces, see e.g.\ \cite{Turaga}.

\subsection{Decompounding in Summary}
In this paper $G$ will always denote a compact and connected Lie group of finite dimension, and $(\Omega, \FF,\MP)$ is a fixed probability space. As studied in \cite{Said}, see also \cite{Applebaum,Applebaum2,Liao} a compound process $(y_t)_{t\geq 0}$ on $G$ is a stochastic process on $G$ defined by
\begin{align*}
    y_t = \prod_{n = 1}^{N_t} x_n,
\end{align*}
for an identically distributed sequence $(x_n )_{n \geq 1}$ of random variables on $G$ that are pairwise independent not only to each other but also to a homogeneous Poisson process $(N_t)_{t\geq 0}$ for all $t \geq 0$. In other words, the variables $x_n$ act as step distributions in a random walk on $G$ whose walking frequency is randomly determined by $(N_t)_{t \geq 0}$. From this, we identify the decompounding problem as the inference problem of estimating the distribution of the step distributions $(x_n)_{n \geq 1}$ given observations of the process $(y_t)_{t \geq 0}$. Compound processes $(p_t)_{t \geq 0}$ on a symmetric space $M = G/K$ take the similar form
$$
p_t = \bigg(\prod_{n= 1}^{N_t}x_n\bigg)K,
$$
a definition that will be explained and motivated in more detail in Definition \ref{def:Compound}. The crucial detail to notice is that $(p_t)_{t \geq 0}$ can be viewed as a random walk that acts on the quotient $G/K$, in contrast to the compound process $(y_t)_{t \geq 0}$ on $G$ which must take place on the space itself.

Compound processes on symmetric spaces have been studied previously in \cite{Applebaum2}, particularly with regards to the asymptotic behaviour of their distributions as $t \to \infty$. However, decompounding methods on symmetric have to the authors best knowledge not previously been established. A first step in this endeavor is therefore to propose a method of decompounding on symmetric spaces. In the initial work on Lie groups of \cite{Said}, a major contribution lied in showing that under mild assumptions on the distribution of $(x_n)_{n \geq 1}$, among them that the distribution admits a probability density, the Fourier coefficients of the probability density can be approximated with convergence in probability. Although, the rate at which it does so was not discussed in detail. It naturally follows to question whether stronger notions of convergence are possible. The first contribution of this paper lies in showing that similar estimators can be constructed for decompounding on a symmetric space, and that the squared error can be strengthened to converge in expectation at an inverse proportional rate to the number of available samples for estimation. This is the statement of Proposition \ref{prop:L1Estimator}.

Note in particular that considerations of symmetric spaces are an extension to the Lie group case in the literal sense. Indeed, the compact Lie groups form a special subclass of the compact symmetric spaces under the identification of a Lie group as symmetric space $G \cong (G \times G)/\Delta G$ by the diagonal subgroup $\Delta G :=\{(x,x)\in G\times G: x \in G\}$ of $G \times G$. 

\subsection{Extending to the Question of Optimal Estimation}
Once Fourier coefficients can be approximated, the question moves to how to choose a finite number of them to include in an approximate series expansion of a probability density. We are thus led to the domain of inference and information theory in the pursuit of making this choice optimal.

%The total expected squared error of such an estimator will split into a variance term and bias term as is typical in inference on function spaces. The second main contribution of this paper is the change of perspective to information theoretical aspects of decompounding. 
Our second contribution therefore lies in characterizing decompounding as a density estimation problem, to establish a method with desirable convergence rates, and to establish how it fares in comparison to lower bounds of density estimation on symmetric spaces. Explicitly, we will view it as a density estimation problem over a suitable Sobolev class of densities. We will show that the total expected mean squared error $\ME\|f_X^{(m)}-f_X\|^2$ of the estimated density $f_X^{(m)}$ to the true density $f_X$ converges at a rate of $$m^{-2s/(2s+d)}$$ for the dimension $d$ of the symmetric space and $s > 0$ the Sobolev class of $f_X$. Here, $m \geq 1$ is the number of observations used in the estimate. The construction of the estimator relies on the harmonic analysis of the symmetric space as well as the properties of its Laplace--Beltrami-operator, which explains the occurrence of $d$ in this expression. This is the statement of Theorem \ref{theorem:densityEstimation}.

The question of whether this estimator is optimal is more intricate. Lower bounds on density estimation on symmetric spaces have been studied quite recently in \cite{Asta1,Asta2} for symmetric spaces of non-compact type, where in particular the identical rate $m^{-2s/(2s+d)}$ is attained. Upper bounds on kernel density estimation have also been studied on arbitrary compact Riemannian manifolds where kernel estimators are constructed that attain identical rates \cite{Pelletier}. Of course, the expression should also be familiar from rates expected from inference problems in the Euclidean space $\real^d$, covered in classical books such as \cite{Tsybakov}. Lower and exact bounds of density estimations have been found on Lie groups for linear deconvolution problems in \cite{KooKim} and also to some extent on the $2$-sphere in \cite{KooKim2}. There, more specific rates are attained by putting assumptions of smoothness on the linear channel, but it should be noted that decompounding is a non-linear deconvolution problem so the exact rates of \cite{KooKim} do not apply in this setting.

Continuing further, there exist results in describing the $n$-width Kolmogorov distances for much more general classes of manifolds on Sobolev spaces defined by more intricate diffusion operators \cite{Geller,Pesenson}. While similar in their results and ideas, the $n$-width Kolmogorov distance assumes deterministic knowledge of finite-dimensional subspaces. Density estimation involves both the choice of subspaces to use for estimation, but also the estimation of elements within these subspaces. 

Our third contribution therefore lies in providing an exhaustive proof for the case that density estimation of densities of uniformly bounded Sobolev class $s > 0$ in a compact symmetric space is lower bounded by the rate 
$$
m^{-2s/(2s+d)}.
$$
This result is to be expected from the analogous result on symmetric spaces of non-compact type, and the local topological equivalence between a symmetric space of dimension $d$ and $\real^d$. It is to the authors best knowledge a result which has no standard reference in the literature, however. The main ideas in the proof are to combine classical arguments by means of constructing a finite well-separated set of densities for an application of Fano's inequality with the geometrical structure of the symmetric space and its harmonic analysis via its Laplace--Beltrami-operator. This is the content of Proposition \ref{prop:lowerBoundM}.

The final contribution of the paper lies in realizing that the lower bound above does not necessarily apply to the decompounding problem. Rather, it falls in the category of a smaller class of density estimation problems of estimators with invariance properties on $M$. We therefore provide a proof that in this subclass, a lower bound exists in 
$$
m^{-2s/(2s+r)}
$$
where $r = \textup{rank}(M)$ is the rank of the symmetric space $M$, see Proposition \ref{prop:lowerBoundMKinvariant} below. In other words, the optimality of the proposed decompounding estimator depends on the rank of the symmetric space and is closer to optimality for large values of $r$.

Consequently, our proposed estimator for decompounding converges at a stronger and faster rate than has previously been established, and the rate follows the asymptotical lower bounds of density estimators on compact symmetric spaces of maximal rank $r = d$. It remains unclear whether the lower bound for compact symmetric space of non-maximal rank $r < d$ is exact for the decompounding problem specifically. If it is the case, there remains a gap for future work in providing an estimator which attains the minimax rates when $r < d$. The symmetric space decompounding method applies to a broader range of problems with non-trivial extensions to symmetric spaces that make common appearances in the statistics literature, both from the applied and theoretical perspective.
\subsection{Outline}
The paper is structured as follows: Section \ref{sec:background} provides a brief account of the known results from the theory of Lévy processes on Lie groups and symmetric spaces, mainly setting the notation of Hunt's generator formula and the Lévy--Khinchin-formula, which is important for characterizing the relation between the distribution of the observations and the latent distribution. In Section \ref{sec:FourierCoeffs}, the statistical assumptions are formulated, which allows for a detailed description about the Fourier coefficient estimators. Section \ref{sec:Decompounding} continues by using the provided Fourier coefficient estimators to construct an estimator of the probability density. We also provide the proofs of the lower bounds on density estimation under the assumptions of the decompounding problem. Section \ref{Sec:FurtherWork} briefly discusses the effect that noise perturbations have on the method, highlighting avenues for future work. Section \ref{sec:Conclusion} ends the paper by providing a conclusion of the results and its implications.

\subsection{Notation and Conventions}
%We adhere to the convention that the natural numbers are positive and denote them by $\MN := \{1,2,3,\ldots\}$. The non-negative integers are denoted by $\MZ_+ := \{0,1,2,3\ldots\}$.
For two real-valued functions (or sequences) $f$ and $g$, the notation $f \lesssim g$ signifies that there exists a constant $C > 0$ for which 
$f \leq Cg$ in the entire domain of the functions (or points of the sequence). Writing $f \sim g$ signifies that $f \lesssim g$ and $g \lesssim f$. When $f\sim g$ we say that $f$ and $g$ are asymptotically equivalent.

Except possibly for the probability space $(\Omega,\FF,\MP)$, all measures and functions that are considered in this paper appear on topological spaces. For brevity, we will always assume that measurability refers to the underlying Borel $\sigma$-field and that the measures are Borel measures. 

Lastly, we reserve the letter $\chi$ to denote indicator functions. That is, if $A \in \FF$ is an event, then
$$
\chi_A(\omega) := \begin{cases}
    1, & \textup{if } \omega \in A\\
    0, & \textup{if } \omega \notin A.
\end{cases}
$$
While a trivial remark, we nonetheless bring it to attention since this nomenclature may be seen as atypical in the probability and statistics literature.

\section{Background}
\label{sec:background}
\subsection{Preliminaries}
We will briefly recollect and set notation for the theory of harmonic analysis of compact Lie groups and symmetric spaces, which can be viewed in more detail in classical references such as \cite{Varadarajan,Helgason,Helgason3} or \cite{Applebaum,Liao} for a more probabilistic approach. For the compact and connected Lie group $G$ we denote by $\lambda$ its normalized probability Haar measure $\lambda(G) = 1$. The collection of equivalence classes of irreducible linear representations of $G$, identified up to isomorphism, is denoted by $\hat{G}$. By standard procedure of mixing by $\lambda$, each class in $\hat{G}$ contains a unitary representation $(\pi, V_\pi)$ for $V_\pi$ a linear space of finite dimension $d_\pi$ and $\pi: G \to \unitary(V_\pi)$, where $\unitary(V_\pi)$ is the group of unitary invertible linear endomorphisms of $V_\pi$. With abuse of language, we will always choose a unitary representation to represent each class in $\hat{G}$ and without further care use notation such as $\pi \in \hat{G}$ for such a unitary representation $\pi$. Recall in particular that $\hat{G}$ is countable, that all irreducible representations of $G$ are finite-dimensional, and that each $\pi \in \hat{G}$ is countably infinitely differentiable $\pi \in C^\infty(G)$.

We will typically write $L^2(G) := L^2(G,\lambda; \MC)$ and denote its inner product $\langle f,g\rangle := \int_Gf\,\overline{g}\,d\lambda$. For an arbitrary Hilbert space $H$ we will denote its inner product by $\isp_{H}$ to not cause confusion. For any representation $\pi \in \hat{G}$, the function $\pi_{ij}: G \to \MC$ which takes the value of the matrix representation of $\pi$ at index $(i,j)$ under some fixed basis of $V_\pi$, for $1 \leq i,j \leq d_\pi$, is a smooth function. By the Peter--Weyl theorem \cite[Corollary 2.2.4]{Applebaum} the set 
$$
\{d_\pi^{1/2}\pi_{ij}: \pi \in \hat{G}, \: 1 \leq i,j \leq d_\pi\}
$$
forms a complete ON-system of $L^2(G)$.

A compact symmetric space will be denoted by $M$, but is identified as the quotient $G/K$ by a closed (stabilizer) Lie subgroup $K$ with an involutive Cartan automorphism $\Theta$. Elements of $M$ are written either as $p \in M$ or simply in coset form $xK$, and the element which corresponds to $eK \in G/K$ is denoted by $p_0$. Furthermore, we set $L^2(M) := L^2(M, \lambda_M; \MC)$, where $\lambda_M = \lambda \circ \sigma^{-1}$ is the translation invariant measure induced by pushing forward by the canonical projection 
$$\sigma: G \to G/K.
$$
The inner product in $L^2(M)$ is also simply denoted by $\isp$. An irreducible representation $\pi$ for which the $K$-fixed point subspace $V_\pi^K := \{v \in V_\pi: \pi(k)v = v,\: \forall k \in K\}$ is non-zero is said to be a spherical representation and we denote the set of such representations by $\hat{G}_K$. Since $V_\pi^K$ is at most of dimension $1$ for compact symmetric spaces $M$, the definition
\begin{align}
\label{eq:spherical}
\phi_\pi: M \to \MC, \qquad \phi_\pi(xK) = \langle \pi(x)v_0,v_0\rangle_{V_\pi}
\end{align}
is well-posed and independent of the choice of any $v_0 \in V_\pi^K$ of unit length. We recall in particular, see e.g.\ \cite[Theorem 5.18]{Liao}, that functions of the form $\phi_\pi$ above, for $\pi \in \hat{G}_K$, constitute the spherical functions on $M$ and that  the set
$$
\{d_\pi^{1/2}\phi_\pi: \pi \in \hat{G}_K\}
$$
forms an ON-system in the closed subspace $L^2_K(M)$ of $K$-invariant functions in $L^2(M)$. 
\subsection{Compound Processes}
Let us formalize the definition of a compound process.
\begin{definition}
\label{def:Compound}
    Let $(x_n)_{n \geq 1}$ be a sequence of i.i.d.\ variables on $G$ and let $(N_t)_{t \geq 0}$ be a homogeneous Poisson counting process with intensity parameter $\Lambda > 0$ which is independent to $x_n$ for all $t \geq 0$ and $n \geq 1$. A \textbf{compound process} on $G$ is a process $(y_t)_{t \geq 0}$ of the form
    \begin{align}
        \label{eq:compoundOnG}
        y_t = \prod_{n = 1}^{N_t} x_n,
    \end{align}
    where the product is taken from left to right.

    A compound process $(p_t)_{t \geq 0}$ on a symmetric space $M$ is a process of the form $p_t = \sigma(y_t)$ for a compound process $y_t$, where in addition the distribution of $x_n$ is assumed to be $K$-bi-invariant. That is,
    \begin{align}
    \label{eq:compoundOnM}
    p_t = \bigg(\prod_{n = 1}^{N_t}x_n\bigg)K,
    \end{align}
    and the distribution of $x_n$ is invariant under both left and right translations by elements in $K$.
\end{definition}
Observe that a compound process $(y_t)_{t\geq 0}$ on $G$ is a Lévy process, which is important for describing the generator of the distributions $\mu_t$ of $y_t$. Lévy processes on $M$ correspond precisely with the projections by $\sigma$ of $K$-bi-invariant Lévy processes on $G$ \cite[Proposition 1.12]{Liao}. This motivates the assumption that the steps $(x_n)_{n \geq 1}$ in Equation \eqref{eq:compoundOnM} are $K$-bi-invariant.

Compound processes on $M$ have also been studied with respect to their asymptotic behaviour in distribution as $t \to \infty$ in \cite{Applebaum2}. As in the case of compact Lie groups, see \cite[Proposition 5]{Said}, the distribution of $p_t$ converges weakly to $\lambda_M$ as $t \to \infty$ at an exponential rate, under mild assumptions on the initial distribution of $(p_t)_{t\geq 0}$. This tells us that information about the distribution of the steps $(x_n)_{n \geq 1}$ dissipates as time passes and is relevant for questions on sampling of $(y_t)_{t \geq 0}$ or $(p_t)_{t \geq 0}$ in the decompounding problem.

\subsection{Lévy generators}
Lévy processes on $G$ are rcll Markov processes and we therefore note that the Markov operators $(\widetilde{P_t})_{t \geq 0}$ defined on $f \in C(G)$ by $\widetilde{P}_tf(x) = \ME[f(xx_t)]$ form a $C_0$-semigroup with a densely defined generator $\widetilde{L}$ say \cite[Chapter 17]{Kallenberg}. Like the generator of a Markov process on $\real^d$, it defines the process in terms of a diffusion term, a deterministic drift, and a contribution by a Lévy measure, as originally proven by Hunt \cite{Hunt}. Explicitly, we will consider generators of Lévy processes on $G$ to be defined by $f \in C^\infty(G)$ as 
$$
\widetilde{L}f(x) = \frac{1}{2}\widetilde{D}f(x)+\xi_0f(x)+\int_G (f(xy)-f(x))\,d\eta(y),
$$
for a deterministic drift vector $\xi_0\in \gg$ on the Lie algebra $\gg$ of $G$, a non-negative definite symmetric and $G$-invariant diffusion operator $\widetilde{D}$, and a Lévy measure $\eta$ on, see also \cite[Theorem 2.2]{Liao}. If $p_t = \sigma(y_t)$ is a Lévy process on $M$, then its generator $L$ is simply given, for $f \in C^\infty(M)$, by
%If we endow $G$ with an adjoint-invariant Riemannian metric, e.g.\ by extension from the Killing form on $\gg$, and consider the corresponding decomposition $\gg = \pp \oplus \hh$ with $\hh$ the Lie algebra of $K$ and $\pp = \hh^\perp \cong T_{eK}M$ its complement, then the generator $L$ of $(\sigma(y_t))_{t \geq 0}$ takes a similar form. Indeed, in that case $L$ is determined for $f \in C^\infty(M)$ by the corresponding Hunt's formula
\begin{align}
    \label{eq:HuntsOnM}
    Lf(xK) = \frac{1}{2}\sigma(\widetilde{D})f(xK) + \sigma(\xi_0)f(xK)+\int_M(f(xyK)-f(xK))\,d\eta(y), 
\end{align}
where $\sigma(\widetilde{D})$ and $\sigma(\xi_0)$ are the lifted operators via pullback by $\sigma$. The fact that $\sigma(\widetilde{D})$ and $\sigma(\xi_0)$ are well-defined can viewed in more detail in \cite[Chapter 3]{Liao}, but comes down to the fact that the $K$-bi-invariance of $(y_t)_{t \geq 0}$ forces $\widetilde{D}$ and $\xi_0$ in $\widetilde{L}$ to have additional invariance properties in $K$.

\subsection{Lévy--Khinchin Formulas}
From now on $(\mu_t)_{t \geq 0}$ denotes the family of distributions $\mu_t := \MP \circ y_t^{-1}$ of a Lévy process $(y_t)_{t \geq 0}$ on $G$ and $(\nu_t)_{t \geq 0}$ the distributions $\nu_t := \MP \circ p_t^{-1}$ of a Lévy process $(p_t)_{t \geq 0}$
on $M$. Since the processes are (left) $G$-invariant, there is no loss in generality in assuming that their initial distributions are fixed at $e \in G$ or $p_0 \in M$. By straightforward computations\footnote{This is seen more directly by noting that the distributions of Lévy processes on either $G$ or $M$ form convolution semigroups.} it is seen that $\mu_{s+t}(\pi) = \mu_s(\pi)\mu_t(\pi)$ and $\mu_0(\pi) = \Id$ for all $s, t \geq0$ and $\pi \in \hat{G}$. In other words
$$
\{\mu_t(\pi): t \geq 0\}
$$
forms a $1$-parameter semigroup of matrices of size $d_\pi \times d_\pi$ and is therefore of the form $\mu_t(\pi) = \exp(tA)$ for some matrix $A$. By differentiating at $t = 0$, it is easy to see that the generator $A$ is equal to $\widetilde{L}\pi(e)$, where $\widetilde{L}\pi = (\widetilde{L}\pi_{ij})_{i,j = 1}^{d_\pi}$. Hence, we obtain a Lévy--Khinchin-type formula in
$$
\mu_t(\pi) = \exp\bigg(t\widetilde{L}\pi(e)\bigg).
$$
 Similarly, $\nu_{s+t}(\phi_\pi) = \nu_s(\phi_\pi)\nu_t(\phi_{\pi})$ for any spherical function $\phi_\pi$ and all $s,t \geq 0$ and by identical argumentation the formula
\begin{align}
    \label{eq:LevyKhinchinM}
    \nu_t(\phi_\pi) = \exp\bigg(tL\phi(eK)\bigg)
\end{align}
is true as well. As Proposition \ref{prop:LevyKhinchinForCompound} below shows, the generator $L$ of a compound process is particularly easily expressible.
\begin{proposition}
    \label{prop:LevyKhinchinForCompound}
    Let $\mu_X$ be the distribution of the steps $(x_n)_{n \geq 1}$ of a compound process $p_t = \sigma(y_t)$ on $M$. Then
    \begin{align}
    \label{eq:LevyKhinchinCompoundM}
    \nu_t(\phi_\pi) = \exp(t\Lambda \mu_X(\phi_\pi\circ \sigma)-t\Lambda).
    \end{align}
    In particular, $(p_t)_{t \geq 0}$ has no drift or diffusion term, and its Lévy measure is given by $\eta = \Lambda \mu_X\circ \sigma^{-1}$.
\end{proposition}
\begin{proof}
    The proof is done in two steps. Firstly, we follow a similar approach to that of \cite{Said} to compute $\mu_t(\pi)$ explicitly. Then, the structure of $\widetilde{L}$ can be identified and lifted to $L$ for an expression of $\nu_t(\phi_{\pi})$ thanks to Equation \eqref{eq:LevyKhinchinM}. Indeed, for $\pi \in \hat{G}$, we condition over the values of $N_t$ to see that
    \begin{align*}
        \mu_t(\pi) &= \ME[\pi(y_t)] \\
        & =\ME[\ME[\pi(y_t)|N_t]]\\
        & = \sum_{n = 1}^\infty \ME[\pi(y_t)| N_t = n]\MP[N_t = n]\\
        & = {e^{-\Lambda t}}\sum_{n = 1}^{\infty}\frac{(\Lambda t)^n}{n!}\ME\Bigg[\pi\bigg(\prod_{m = 1}^n x_m\bigg)\Bigg].
    \end{align*}
    However, mutual independence in the step distributions gives
    $$
    \ME\Bigg[\pi\bigg(\prod_{m = 1}^n x_m\bigg)\Bigg] = \prod_{m = 1}^n\ME[\pi(x_m)] = \mu_X(\pi)^n.
    $$
    Hence,
    $$
    \mu_t(\pi) = e^{-t\Lambda }\sum_{n = 1}^\infty \frac{(t\Lambda )^n}{n!}\mu_X(\pi)^n = e^{-t\Lambda }\exp(t\Lambda \mu_X(\pi)) = \exp\Big(t\Lambda(\mu_X(\pi)-\Id)\Big) = \exp\Big(t\Lambda(\mu_X(\pi)-\pi(e))\Big).
    $$
    This shows that $\widetilde{L}\pi(e) = \Lambda \mu_X(\pi)-\Lambda \pi(e)$ and completes the first step. Because of the Peter--Weyl theorem the values of $\widetilde{L}$ for $\pi \in \hat{G}$ determine the generator, and it must therefore be given by
    $$
    \widetilde{L}f(x) = \Lambda\big(\mu_X(y \mapsto f(xy))-f(x)\big)
    $$
    for a general $f \in C^\infty(G)$. This expression is easily rewritten as
    $$
    \widetilde{L}f(x) = \Lambda \int_Gf(xy)-f(y)\,d\mu_X(x)
    $$
    which identifies the Lévy measure $\eta$ of $y_t$ as $\eta = \Lambda \mu_X$. It also shows that $(y_t)_{t \geq 0}$ has no drift or diffusion. The Lévy measure $\eta_M$ of $(p_t)$ is therefore $\eta_M = \Lambda \mu_X\circ \sigma^{-1}$ and the expression
    $$
    \nu_t(\phi_\pi) = \exp(t\Lambda \mu_X(\phi_\pi\circ \sigma)-t\Lambda)
    $$
    follows.
\end{proof}

\section{Fourier Coefficient Estimators}
\label{sec:FourierCoeffs}
\subsection{Statistical Assumptions}
A pressing question is in which sense information about $(p_t)_{t \geq 0}$ is to be assumed known. We will assume that a finite number $m \geq 1$ of independent observations of $p_t$ are available, for a fixed time $t > 0$. Note that independent sampling of this sort does not necessarily assume that observations have been made from different independent trajectories. By the stationarity and independence of the increments, it also covers the case when observations $p_{t_1},p_{t_2},\ldots,p_{t_m}$ have been made of a single trajectory $(p_t)_{t \geq 0}$, as long as the timesteps $t_{k}-t_{k-1}$ for $2 \leq k \leq m$ remain constant, which may be seen as a more reasonable assumption depending on the application. It would certainly be desired to develop ways of decompounding for more general sampling assumptions, but this makes the problem significantly more difficult. While the Lévy--Khinchin formula provides a relation between randomly sampled observations, it is ostensibly a situation where we are dealing with $m$ observations of $m$ different distributions. Decompounding methods in this situation is admittedly a gap in research which would be interesting to investigate, but for the purposes of this paper, we restrict attention to the following.
\begin{assumption}
\label{assumption:lowfrequency}
    Assume that there exists a sequence $(p_m)_{m \geq 1}$ of independent observations of the compound process $p_t$ observed at a fixed time $t > 0$.
\end{assumption}
We acknowledge the unfortunate clash in notation between the observations $(p_m)_{m \geq 1}$ of $p_t$ for $t$ fixed and the stochastic process $(p_t)_{t \geq 0}$, but this should not pose an issue when read in context.

Secondly, assumptions on the distribution of the steps $(x_n)_{n \geq 1}$ are needed. We will assume that the distribution of $x_n$ is sufficiently regular for it to admit a probability density $f_X \in L^2(G)$ with respect to the Haar measure $\lambda$. Necessary and sufficient conditions for this to hold is a classically studied problem and can be read about in more detail for example in \cite{Applebaum}. Since each $x_n$ is $K$-bi-invariant, it is easy to see that $f_X$ satisfies $f_X(xk) = f_X(kx) = f_X(x)$ for all $k \in K$. With slight abuse of notation, we therefore identify $f_X \in L^2_K(M)$ by setting $f_X(xK) = f_X(x)$.
\begin{assumption}
\label{assumption:density}
    The steps $(x_n)_{n \geq 1}$ are assumed to admit a $K$-bi-invariant probability density $f_X \in L^2(G)$.
\end{assumption}
Furthermore, it can be useful to assume that the steps $x_n$ are inverse invariant in distribution, i.e.\ $x_n \overset{d}{=} x_n^{-1}$. This is not a great restriction since many commonly appearing distributions are inverse invariant, but computations become much simpler as Proposition \ref{prop:LevyKhinchinWellPosed} below shows. We will return to the loosening of this assumption in Section \ref{subsec:inverseInvariance}. For now however, we assume the following.
\begin{assumption}
\label{assumption:Invariant}
The steps $(x_n)_{n\geq 1}$ are assumed to be inverse invariant, whence their probability density $f_X$ satisfies
$$
f_X(x)= f_X(x^{-1})
$$
for all $x \in G$.
\end{assumption}
As a consequence of Assumptions \ref{assumption:lowfrequency} and \ref{assumption:density}, the problem of decompounding can be seen as an inference problem in $L_K^2(M)$. Because of this, our first steps towards an inference method is to develop estimators of the Fourier coefficients $\langle f_X,\phi_\pi\rangle$ for $\pi \in \hat{G}_K$.

%In this paper, decompounding therefore refers to the problem of inferring $f_X$ under Assumptions \ref{assumption:lowfrequency} and \ref{assumption:densityandInvariant}. An approximation $f_X^{(m)} \in L_K^2(M)$ which uses only the $m$ first observations will be shown to satisfy that the error $\|f_X^{(m)}-f_X\|_{L^2(M)}$ converges to zero in probability and at a rate which depends only on intrinsic properties of $M$, see Theorem \ref{thm:decompounding} below.
\subsection{Inverse Invariant Estimation of the Fourier Coefficients}

With Proposition \ref{prop:LevyKhinchinForCompound}, a straightforward approach of estimating the distribution $\mu_X$ appears by performing the formal calculations
\begin{align}
\label{eq:inverseOfCompound}
\mu_X(\phi) = \langle f_X,\phi\rangle= \frac{1}{t\Lambda}\log(\nu_t(\phi))+1.
\end{align}
In turn, the inference problem is reduced to the typical scenario of estimating coefficients in a series expansion $f_X = \sum_{\pi \in \widehat{G}_K} d_\pi \langle f_X,\phi_\pi\rangle\phi_\pi$. Let us therefore motivate why this derivation is justified, and establish how to estimate $\nu_t(\phi)$. The latter point is easily achieved by observing that that the empirical averages
\begin{align}
\label{eq:empiricalAverage}
\nu_{m}(\pi) := \frac{1}{2m}\sum_{k = 1}^m\bigg(\phi_\pi(p_k)+\overline{\phi_\pi(p_k)}\bigg)
\end{align}
converge in mean squared error to $\nu_t(\phi_\pi)$ and the rate at which they do so is of order $m^{-1}$ as $m \to \infty$, where $\overline{z}$ denotes the complex conjugate of $z \in \mathbb{C}$. Point-wise almost sure convergence of \eqref{eq:empiricalAverage} is an immediate consequence of the law of large numbers, but it is the fact that the variance of $\phi_\pi(p_k)$ satisfies $$\MV[\phi_\pi(p_k)] := \ME[|\phi_\pi(p_k)|^2]-|\ME[\phi_\pi(p_k)]|^2 \leq 1$$
which cements the fact that the convergence is in mean squared error at the rate $m^{-1}$. In fact, the boundedness of $\frac{(\phi_\pi(p_k)+\overline{\phi_\pi(p_k)})}{2}$ in $[-1,1]$ will be repeatedly useful in proofs to come.

As Proposition \ref{prop:LevyKhinchinWellPosed} below shows, Equation \eqref{eq:inverseOfCompound} is justified formally.
\begin{proposition}
    \label{prop:LevyKhinchinWellPosed}
    Under Assumptions \ref{assumption:lowfrequency}, \ref{assumption:density}, and \ref{assumption:Invariant}, the integrals $\nu_t(\phi_\pi)$ are uniformly positive. That is,
    $$
    \inf_{\pi \in \hat{G}_K}\nu_t(\phi_{\pi}) > 0.
    $$
\end{proposition}
\begin{proof}
    In view of Equation \eqref{eq:LevyKhinchinCompoundM}, it suffices to show that $\langle f_X,\phi_\pi\rangle$ is real to show that $\nu_t(\phi_\pi)$ is positive. Note that since $\pi(x)$ is unitary for all $x \in G$,
    $$
    \langle \pi(x^{-1})v_0,v_0\rangle = \langle \pi(x)^*v_0,v_0\rangle = \langle v_0,\pi(x)v_0\rangle = \overline{\langle \pi(x)v_0,v_0\rangle}.
    $$
    In other words, $\overline{\phi_\pi(xK)} = \phi_\pi(x^{-1}K)$. Combining this fact with the inverse invariance of the Haar measure $\lambda$, the following computations are justified:
    \begin{align*}
        \langle f_X,\phi_\pi\rangle& = \int_Gf_X(x)\overline{\phi_\pi(x^{}K)}\,d\lambda(x)\\
        & = \int_G f_X(x^{-1})\overline{\phi_\pi(x^{-1}K)}\,d\lambda(x)\\
        & = \int_G f_X(x)\phi_\pi(xK)\,d\lambda(x)\\
        & = \langle \phi_\pi,f_X\rangle,
    \end{align*}
    where the last equality follows from the trivial observation that $f_X$ is real-valued. The fact that $\nu_t(\phi_\pi)$ preserves a uniform distance to zero is also a simple consequence of Equation \eqref{eq:LevyKhinchinCompoundM}: For any limit in $\hat{G}_K$, the integrals $\nu_t(\phi_\pi)$ can only tend to $0$ if $\langle f_X,\phi_\pi\rangle$ tends to $1$, but this is manifestly not the case since $f_X \in L_K^2(M)$ has a convergent series expansion in the sphericals $\phi_\pi$.
\end{proof}
With Proposition \ref{prop:LevyKhinchinWellPosed}, we are ready to construct an estimator of $\langle f_X,\phi_\pi\rangle$.
\begin{proposition}
\label{prop:L1Estimator}
    Let $\delta >0$. The estimators
    $$
    c_m(\pi) := \begin{cases}
        \frac{1}{t\Lambda}\log(\nu_m(\pi))+1 & \textup{if } \nu_m \geq \frac{\delta}{m}\\
        0 & \textup{otherwise}
    \end{cases}
    $$
    converge in expectation towards $\langle f_X,\phi_\pi\rangle$. Explicitly, there exist constants $c,C_1,C_2,C_3 > 0$ such that the bound
    \begin{align}
    \label{eq:UniformBound}
    \ME[|c_m(\pi)-\langle f_X,\phi_\pi\rangle|^2] \leq C_1\frac{1}{{m}}+C_2\log(m)^2e^{-cm}+C_3e^{-cm}
    \end{align}
    holds uniformly for varying $\pi \in \hat{G}_K$
\end{proposition}
\begin{proof}
    The main property that will be utilized is the fact that $\frac{(\phi_\pi(p_k)+\overline{\phi_\pi(p_k)})}{2}$ are real random variables that are bounded by the interval $[-1,1]$. The truncating sequence $(\frac{\delta}{m})_{m \geq 1}$ is introduced to have control over the behaviour of events where $\log(\nu_m)$ is singular.

    For brevity, let $\phi = \phi_\pi$ be fixed and write $\nu(\phi) := \nu_t(\phi)$. Let $U_m$ be the event that $\nu_m \geq \frac{\nu(\phi)}{2}$, which without loss of generality may be assumed to be contained in the event that $\nu_m \geq \frac{\delta}{m}$. Then
    \begin{align}
    \label{eq:eqForProp}
    \ME[|c_m-\langle f_X,\phi\rangle|^2] = \int_{U_m}|c_m-\langle f_X,\phi\rangle|^2\,d\MP + \int_{\frac{\delta}{m} \leq \nu_m<\frac{\nu(\phi)}{2}}|c_m-\langle f_X,\phi\rangle|^2\,d\MP + \int_{\nu_m < \frac{\delta}{m}}|\langle f_X,\phi\rangle|^2\,d\MP.
    \end{align}
    For $a,b \geq \frac{\nu(\phi)}{2}$ we have the Lipschitz bound
    $$
    |\log(a)-\log(b)|\leq\frac{2}{\nu(\phi)}|a-b|,
    $$
    whence
    $$
    \int_{U_m}|c_m-\langle f_X,\phi\rangle|^2\,d\MP \leq \Big(\frac{2}{t\Lambda \nu(\phi)}\Big)^2\int_{U_m}|\nu_m-\nu(\phi)|^2\,d\MP.
    $$
    This tends to zero of order $\frac{1}{m}$ as $m \to \infty$ since $\nu_m$ is the empirical average of variables with at most unit variance. The second integral appearing in \eqref{eq:eqForProp} can be rewritten as
    $$
    \int_{\frac{\delta}{m} \leq \nu_m<\frac{\nu(\phi)}{2}}|c_m-\langle f_X,\phi\rangle|^2\,d\MP = \frac{1}{(t\Lambda)^2}\int_{\frac{\delta}{m} \leq \nu_m < \frac{\nu(\phi)}{2}}|\log(\frac{\nu_m}{\nu(\phi)})|^2\,d\MP \leq \frac{1}{(t\Lambda)^2}|\log(\frac{\delta}{\nu(\phi)m})|^2\MP(\frac{\delta}{m} \leq \nu_m < \frac{\nu(\phi)}{2}).
    $$
    The probability $\MP(\frac{\delta}{m} \leq \nu_m < \frac{\nu(\phi)}{2}) \leq \MP(\nu_m < \frac{\nu(\phi)}{2})$ is approximated by Hoeffding's inequality as
    $$
    \MP(\nu_m-\nu(\phi) < -\frac{\nu(\phi)}{2}) \leq \exp\bigg(-\frac{\nu(\phi)^2}{16}m\bigg)
    $$
    which, again, is applicable since the summands $\frac{(\phi_\pi(p_k)+\overline{\phi_\pi(p_k)})}{2}$ in $\nu_m$ are bounded in $[-1,1]$. In particular, the exponential decay dominates $|\log(\frac{1}{m})|^2$ as $m$ tends to infinity, so the second term also tends to zero. The last integral $|\langle f_X,\phi\rangle|^2 \MP(\nu_m < \frac{\delta}{m})$ that appears in Equation \eqref{eq:eqForProp} is easily seen to converge to zero at an exponential rate by similar approximation.

    The fact that a uniform bound such as Equation \eqref{eq:UniformBound} exists is seen by further bounding the above inequalities using that $\inf_{\pi \in \hat{G}_K} \nu(\phi_\pi)> 0$. To derive the exact form of the bound is simply a matter of elementary computations.
\end{proof}
The truncating sequence $(\frac{\delta}{m})_{m \geq 1}$ used to handle the singularity of $\log$ close to zero is needed for the proof above. However, convergence in probability can be preserved by the more basic estimator which only truncates the value to zero if $\log(\nu_m)$ is not defined. A similar estimator is constructed in \cite{Said} for Lie groups, where the authors apply Chebychev's inequality to argue for the convergence. The convergence rate can be made stronger by similar argumentation to that of Proposition \ref{prop:L1Estimator}, however. 
\begin{proposition}
\label{prop:ConvergenceInProbability}
    The estimator
    $$
    c_m(\pi) := \begin{cases}
        \frac{1}{\Lambda t}\log(\nu_m(\pi))+1 & \textup{if } \nu_m(\pi) > 0,\\
        0& \textup{if } \nu_m(\pi) \leq 0,
    \end{cases}
    $$
    converges in probability towards $\langle f_X,\phi_\pi\rangle$ exponentially fast as $m \to \infty$ and it does so uniformly in $\pi \in \hat{G}_K$. That is, for all $\eps > 0$ and any $\pi \in \hat{G}_K$, 
    $$
    \MP(|c_m(\pi)-\langle f_X,\phi_\pi\rangle|> \eps) \leq Ce^{-cm}
    $$
    for some positive constants $c,C > 0$ that depend on $\eps$ but not on $\pi$.
\end{proposition}
\begin{proof}
    Again, let $\phi = \phi_\pi$ and $\nu = \nu_t$ for brevity. The idea of the proof is to mimic that of the Lie group case, see \cite[Proposition 6]{Said}. Introduce the event
    $$
    U_m := \{\nu_m > \frac{\nu(\phi)}{2}\}.
    $$
    like before. Note that
    $$
    \MP(|c_m-\langle f_X,\phi\rangle |> \eps) \leq \MP(U_m^c)+ \MP(\{|c_m-\langle f_X,\phi\rangle|> \eps\}\cap U_m).
    $$
    The probability $\MP(U_m^c)$ can be bounded above by an application of the Hoeffding's inequality since
    $$
    \MP(\nu_m \leq \frac{\nu(\phi)}{2}) = \MP(\nu_m-\nu(\phi) \leq -\frac{\nu(\phi)}{2}) \leq \exp\bigg(-\frac{\nu(\phi)^2}{16}m\bigg).
    $$
    In $U_m$, we utilize the Lipschitz continuity of $\log$ in $[\frac{\nu(\phi)}{2},+\infty)$ as in the proof of Proposition \ref{prop:L1Estimator} to see that 
    \begin{align*}
        |c_m-\langle f_X,\phi\rangle| & = \frac{1}{\Lambda t}|\log(\nu_m)-\log(\nu(\phi))|\\
        & \leq \frac{2}{t\Lambda \nu(\phi)}|\nu_m-\nu(\phi)|.
    \end{align*}
    This inequality helps us derive the following upper bound on the term $\MP(\{|c_m-\langle f_X,\phi\rangle|> \eps\}\cap U_m)$ that appears above: 
    \begin{align*}
       \MP(\{|c_m-\langle f_X,\phi\rangle|> \eps\}\cap U_m)
        & \leq \MP(\frac{2}{t\Lambda\nu(\phi)}|\nu_m-\nu(\phi)| > \eps)\\
        & = \MP(|\nu_m-\nu(\phi)| > \frac{\eps\Lambda t\nu(\phi)}{2})\\
        & \leq 2\exp\bigg(-\frac{(\eps \Lambda t \nu(\phi)) ^2}{16}m\bigg),
    \end{align*}
    where in the last inequality, a two-sided Hoeffding's inequality is applied. Arguing from Proposition \ref{prop:LevyKhinchinWellPosed} yet again, it is easy to see that the infimum among $\nu(\phi_\pi)$ in $\hat{G}_K$ provide uniform upper bounds that preserve the exponential decay.
\end{proof}
As is seen in the proof above, the constant $C > 0$ for which $\MP(|c_m(\pi)-\langle f_X,\phi_\pi\rangle|> \eps) \leq Ce^{-cm}$ certainly satisfies $C \leq 3$. Of course, this is a rather crude bound, but it can nonetheless be used to bound the actual probability that the error threshold $|c_m(\pi) - \langle f,\phi_\pi\rangle| > \eps$ has been trespassed.

\subsection{Estimation Without Inverse Invariance}
\label{subsec:inverseInvariance}
Assuming that the distribution of $f_X$ is inverse invariant is beneficial in simplifying computational aspects since it implies that $\log(\nu(\phi))$ is well-defined as a real logarithm by Proposition \ref{prop:LevyKhinchinWellPosed}. Although this is not a mathematically essential assumption. Indeed, the values of
$$
\nu_t(\phi_\pi) = e^{t\Lambda \langle f_X,\phi_\pi\rangle}e^{-t\Lambda}
$$
will eventually stay within a complex neighbourhood of $e^{-t\Lambda}$ contained entirely on the positive real half plane. In other words, the principal branch logarithm $\Log(\nu_t(\phi_\pi))$ is eventually well-defined to give
$$
\langle f_X,\phi_\pi\rangle = \frac{1}{t\Lambda}\Log(\nu_t(\phi_\pi))+1.
$$
It is more or less clear how to extend the estimator of Proposition \ref{prop:L1Estimator} to this case, by truncating if the empirical average
$$
\nu_m(\pi) = \frac{1}{m}\sum_{k = 1}^m\phi_\pi(p_k)
$$
does not lie sufficiently close to $e^{-t\Lambda}$. Since $\phi_\pi$ remains bounded in the unit disc, we can still apply Hoeffding inequalities as in the proof of Proposition \ref{prop:L1Estimator} and the convergence rates are identical. We summarize these findings and complete the proof as follows.
\begin{proposition}
\label{prop:L1EstimatorInvarianceRemoved}
    Under Assumptions \ref{assumption:lowfrequency} and \ref{assumption:density}, let 
    $$
    \nu_m(\pi) = \frac{1}{m}\sum_{k = 1}^m \phi_\pi(p_k)
    $$
    and consider the estimator
    $$
    c_m(\pi) := \begin{cases}
        \frac{1}{t\Lambda}\Log(\nu_m(\pi))+1, & \textup{if } \mathfrak{Re}(\nu_m)>0 \textup{ and } |\nu_m(\pi)| \geq \frac{\delta}{m}\\
        0, & \textup{otherwise}
    \end{cases}
    $$
    for a fixed $\delta > 0$. Then 
    $$
    \ME|c_m(\pi)-\langle f_X,\phi_\pi\rangle|^2\leq  C \frac{1}{m}
    $$
    for a constant $C$ which does not depend on $\pi \in \hat{G}_K$.
\end{proposition}
\begin{proof}
    The idea of the proof is the same as that of Proposition \ref{prop:L1Estimator}. We will introduce an event $U_m$ where the estimator has good convergence properties, an event $V_m$ disjoint to $U_m$ where the truncation at $|\nu_m| \geq \frac{\delta}{m}$ handles the singularity of $\Log$ with care. The probability of the complement $(U_m \cup V_m)^c$ will then be shown to be sufficiently small to not influence the rate of convergence in $U_m$.
    
    Introduce the event $U_m$ defined as the set of outcomes for which $\nu_m$ has positive real part and that $|\nu_m| > \frac{|\nu(\phi)|}{2}$. In $U_m$ we have the Lipschitz bound
    $$
    |c_m-\langle f_X,\phi\rangle|^2 = (\frac{1}{t\Lambda})^2|\Log(\nu_m)-\Log(\nu(\phi))|^2 \leq (\frac{2}{t\Lambda\nu(\phi)})^2|\nu_m-\nu(\phi)|^2.
    $$
    From this, it is clear that $\ME[ \chi_{U_m}\cdot |c_m-\langle f_X,\phi\rangle|^2]$ tends to $0$ at a rate of $m^{-1}$. Proceeding with the remaining possible outcomes, define $V_m$ as the intermediate set of points that are not in $U_m$, but where the estimator $c_m$ nonetheless is not truncated to zero. That is,
    $$
    V_m = \{\nu_m\notin  U_m\}\cap \{ |\nu_m| \geq \frac{\delta}{m}\} \cap \{\Re(\nu_m) > 0\}
    $$
    consists of the events where $\frac{\delta}{m} \leq |\nu_m| \leq \frac{|\nu(\phi)|}{2}$ and where $\nu_m$ has strictly positive real part. Then,
    $$
    \ME[\chi_{V_m}\cdot |c_{m}-\langle f_X,\phi\rangle|^2] = 
    (\frac{1}{t\Lambda})^2\int_{V_m}|\Log(\frac{\nu_m}{\nu(\phi)})|^2\,d\MP \leq (\frac{1}{t\Lambda})^2|\Log(\frac{\delta}{\nu(\phi)m})|\MP(|\nu_m|\leq \frac{|\nu(\phi)|}{2})
    $$
    where we note that these inequalities hold only because $\nu(\phi)$ and $\nu_m$ are assumed to be part of the positive real half plane. In the final inequality we have also used the fact that $\MP(V_m) \leq \MP(|\nu_m|\leq \frac{|\nu(\phi)|}{2})$.

    Now, $|\nu_m| \leq \frac{|\nu(\phi)|}{2}$ holds if and only if $\Re(\nu_m)^2+\Im(\nu_m)^2 \leq \frac{|\nu(\phi)|^2}{4}$ which in particular implies that $\Re(\nu_m), \Im(\nu_m) \leq \frac{|\nu(\phi)|}{2}$. Hence,
    $$
    \MP(|\nu_m|\leq \frac{|\nu(\phi)|}{2}) \leq \MP(\Re(\nu_m) < \frac{|\nu(\phi)|}{2})+\MP(\Im(\nu_m) < \frac{|\nu(\phi)|}{2})
    $$
    and these probabilities are easily seen to converge to zero at an exponential rate, uniformly in $\pi$ by identical argumentation to the proof of Proposition \ref{prop:L1Estimator}.

    Lastly, the probability of the event $(U_m \cup V_m)^c$ is given by $\MP(\Re(\nu_m) \leq 0 \textup{ or } |\nu_m| < \frac{\delta}{m})$
    and is evidently smaller than the probability $\MP(\Re(\nu_m) < \frac{|\nu(\phi)|}{2})$ which tends to zero exponentially fast by the above argument. Putting the established inequalities together, we obtain that
    $$
    \ME|c_m-\langle f_X,\phi\rangle|^2 \leq \ME[\chi_{U_m}\cdot |c_m-\langle f_X,\phi\rangle|^2]+\ME[\chi_{V_m}\cdot |c_m-\langle f_X,\phi\rangle|^2] + |\langle f_X,\phi\rangle|^2\MP[(U_m\cup V_m)^c] \lesssim m^{-1}
    $$
    holds uniformly in $\hat{G}_K$.
\end{proof}

\section{Decompounding Estimators}
\label{sec:Decompounding}
Propositions \ref{prop:L1Estimator} and \ref{prop:ConvergenceInProbability} establish estimates of the Fourier coefficients of $f_X$. To estimate $f_X$, we also need to treat the number of terms to include in a series expansion to ensure that the errors do not compound in the limit. The natural choice is to let
\begin{align}
    \label{eq:mainEstimator}
    f_X^{(m)} := \sum_{\pi \in F_m}d_\pi c_m(\pi)\phi_\pi
\end{align}
be an estimator of $f_X$ for a non-decreasing sequence $(F_m)_{m \geq 1}$ of non-empty subsets of $\hat{G}_K$ whose union is $\hat{G}_K$. We will call such a sequence an absorbing sequence in $\hat{G}_K$. Then, the error in $L^2$-norm $\|f_X^{(m)}-f_X\|^2$ can be decomposed in the typical fashion 
$$
\|f_X^{(m)}-f_X\|^2 = \sum_{\pi \in F_m}d_\pi|c_m(\pi)-\langle f_X,\phi_\pi\rangle|^2 + \sum_{\pi \notin F_m}d_\pi|\langle f_X,\phi_\pi\rangle|^2 = V_{m} + B_{m}
$$
into a variance term $V_{m} := \sum_{\pi \in F_m}d_\pi|c_m(\pi)-\langle f_X,\phi_\pi\rangle|^2$ and a bias term $B_m := \sum_{\pi \notin F_m}d_\pi|\langle f_X,\phi_\pi\rangle|^2$.
 This section will provide bounds on the growth of absorbing sequences for which convergence of the estimator $f_X^{(m)}$ is preserved.

We fix an adjoint-invariant Riemannian metric on $G$, e.g.\ one induced from the Killing form \cite[Chapter II]{Humphreys}, and consider its corresponding Laplace--Beltrami-operator $\Delta$ on $G$. It lifts to a Laplacian on $M$ via pullback by $\sigma$. With abuse of notation, we shall denote both of these Laplace--Beltrami-operators by $\Delta$. In particular, we recall that $\Delta$ is a nuclear, symmetric, and negative definite operator with a basis of eigenvectors, among them the spherical functions $\{\phi_\pi: \pi \in \hat{G}_K\}$. While there are more eigenfunctions than $\phi_\pi$, the spectrum of $\Delta$ consists only of the eigenvalues of the spherical functions and constitutes the Casimir Spectrum $\textup{Sp}_C(G) := \{\kappa_\pi: \Delta\phi_\pi = -\kappa_\pi\phi_\pi, \:\pi \in \hat{G}_K\}$. We also note that the multiplicity of the eigenspace of $\kappa_\pi$ is $d_\pi$. The fractional Laplace--Beltrami-operator $\Delta^s$ can therefore be defined on $L^2_K(M)$ as
$$
\Delta^sf := \sum_{\pi \in \hat{G}_K}d_\pi (-\kappa_\pi)^s\langle f,\phi_\pi\rangle\phi_\pi
$$
for $s > 0$. The Sobolev space $H^s(M)$ is defined as the completion of $C^\infty(M)$ under the norm
$$
\|f\|_{H^s(M)}^2 = \|f\|^2+\|\Delta^{s/2}f\|^2.
$$
The introduction of $H^s(M)$ allows for a natural choice of absorbing sequence, namely to let $F_m$ be the set
$$
F_m := \{\pi \in \hat{G}_K :\kappa_\pi \leq T_m\}
$$
for a smoothing sequence $(T_m)$, i.e.\ a non-decreasing sequence in $\real$ for which $T_m \to \infty$ as $m \to \infty$. This is a useful definition since the growth of the quantity
$
\sum_{\pi \in F_m} d_\pi
$
is well-known. Indeed, let $(e_k)_{k = 1}^\infty$ be an ON-basis in $L^2(M)$ of eigenvectors to $\Delta$ with corresponding eigenvalues $(\lambda_k)_{k =1 }^{\infty}$, then the classical Weyl-law \cite{Helgason} states that $$
|\{e_k: \lambda_k \leq T\}| \sim T^{\frac{d}{2}}.
$$
However, each eigenvalue $\lambda_k$ is equal to $\kappa_\pi$ for some $\pi \in \hat{G}_K$ and appears with multiplicity $d_\pi$. This means that
$$
\sum_{\pi \in F_m} d_{\pi} = |\{e_k: \lambda_k\leq T_m\}| \sim T_m^{d/2}.
$$

\begin{theorem}
\label{theorem:densityEstimation}
    Suppose that $f_X \in H^s(M)$ for some $s > 0$. If the absorbing sequence
    $$
    F_m :=\{\pi \in \hat{G}_K: \kappa_\pi \leq T_m\}
    $$
    is applied to the estimator in \eqref{eq:mainEstimator} for a smoothing sequence $T_m$ of order $m^{2/(2s+d)}$, then the mean squared error $\ME\|f_X^{(m)}-f_X\|^2$ is of order $m^{-2s/(2s+d)}$.
\end{theorem}
\begin{proof}
    With Proposition \ref{prop:L1Estimator} or \ref{prop:L1EstimatorInvarianceRemoved}, we can find a constant $C > 0$ for which $\ME|c_m(\pi)-\langle f_X,\phi_\pi\rangle|^2 \leq Cm^{-1}$ for all $\pi \in \hat{G}_K$ and all $m \geq 1$. It is therefore clear that
    $$
    \ME[V_{m}] = \sum_{\pi \in F_m}d_\pi\ME|c_m(\pi)-\langle f_X,\phi_\pi\rangle|^2 \leq Cm^{-1}\sum_{\pi \in F_m}d_\pi \lesssim m^{-1}T_m^{d/2}.
    $$
    The property that $\kappa_\pi \leq T_m$ in $F_m$ controls the contribution of the bias term $B_m$ since
    $$
    B_m = \sum_{\pi \notin F_m}d_\pi|\langle f_X,\phi_\pi\rangle|^2 \leq T_m^{-s}\sum_{\pi \notin F_m}d_\pi\kappa_\pi^s |\langle f_X,\phi_\pi\rangle|^2 \leq T_m^{-s}\|\Delta^{s/2}f_X\|^2 \leq T_m^{-s}\|f_X\|_{H^s(M)}^2.
    $$
    Proceeding \textit{ad hoc}, the bias and variance therefore contribute to the total error at comparable rates if $m^{-1}T_m^{d/2} \sim T_m^{-s}$, or equivalently if 
    $$
    T_m \sim m^{2/(2s+d)}.
    $$
    In that case, we see that
    \begin{align*}
    \ME\|f_X^{(m)}-f_X\|^2 &\lesssim  Cm^{-1}T_m^{d/2} + T_m^{-s}\|f_X\|_{H^s(M)}^2 \sim m^{-2s/(2s+d)}
    \end{align*}
    as desired.
\end{proof}
\subsection{Lower Bounds for Density Estimation}
The rate of convergence $m^{-2s/(2s+d)}$ in Theorem \ref{theorem:densityEstimation} corresponds precisely to the lower asymptotic bound of density estimators for inference problems in the Euclidean setting of $M = \real^d$ known in the standard literature \cite{Tsybakov}. Recent work in \cite{Asta1,Asta2} shows that the same lower bound holds true on symmetric space of non-compact type and it would therefore be reasonable to expect that similar bounds hold true in the analytically much simpler case when both $M$ and $G$ are compact. General lower bounds on the $n$-linear widths in compact symmetric spaces have been studied extensively in for example \cite{Geller,Pesenson}. However, the non-deterministic setting of density estimation seems, to the best knowledge of the author, either to be missing or at least to not have a standard reference in the literature. This section will therefore provide the necessary details to draw formal conclusions.

Let $H^s(M,Q)$ denote the probability densities in $L^2(M)$ of Sobolev norm less than or equal to a fixed $Q > 1$, i.e.\
$$
H^s(M,Q) := \{f \in H^s(M): \int f\,d\lambda_M = 1,\: \|f\|^2+\|\Delta^{s/2}f\|^2 < Q^2\}.
$$
Proposition \ref{prop:lowerBoundM} below provides a lower bound on density estimators of elements of $H^s(M,Q)$. 
\begin{proposition}
\label{prop:lowerBoundM}
    Let $\Theta_m$ be the set of estimators $\hat{f}$ of some $f \in H^s(M,Q)$ given $m$ i.i.d.\ observations of $f$. If $s > d/2$, then
    $$
    \inf_{\hat{f} \in \Theta_m} \sup_{f\in H^s(M,Q)} \ME\|\hat{f}-f\|^2 \geq Cm^{-2s/(2s+d)}
    $$
    for some $C >0$.
\end{proposition}
\begin{proof}
    We introduce a complete ON-basis $(e_k)_{k = 1}^\infty$ of $L^2(M)$ consisting of eigenvectors $\Delta e_k = -\lambda_ke_k$ to the Laplace--Beltrami-operator $\Delta$, and define the set $\Gamma_T := \{e_k : \lambda_k \leq T\}$ which follows the growth
    $$
    |\Gamma_T| \sim T^{\frac{d}{2}}
    $$
    by the Weyl counting law of $\Delta$.
    The proof follows the ideas that are employed in the Euclidean case, namely to construct a finite subset $H_m \subset H^s(M,Q)$ of sufficiently large size, where elements are well-separated but also have sufficiently low pairwise Kullback--Liebler-divergence. When this is done, a Fano type inequality can be applied to show that $\ME\|\hat{f}-f\|^2$ is bounded below independently of the choice of estimator. To this end, let $\eps > 0$ and consider the finite set
    $$
    \{f_\alpha = 1+\eps \sum_{e_k \in \Gamma_T}\alpha_ke_k: \alpha = (\alpha_k) \in \{-1,1\}^{\Gamma_T}\}.
    $$
    We pick a subset $H_m$ satisfying $|H_m| \geq 2^{|\Gamma_T|/8}$ wherein the set of coefficients $\alpha,\beta \in \{-1,1\}^{\Gamma_T}$ of any two distinct $f_\alpha \neq f_\beta$ have a Hamming distance
    $$
    d(\alpha,\beta) := |\{k = 1,2,3,\dots,|\Gamma_T|: \alpha_k \neq \beta_k\}|
    $$
    which is greater than $\frac{|\Gamma_T|}{8}$. Hence, distinct $f_\alpha,f_\beta \in H_m$ are well-separated as
    $$
    \|f_\alpha-f_\beta\|^2  = \eps^2\sum_{e_k\in \Gamma_T}|\alpha_k-\beta_k|^2 \geq \frac{\eps^2|\Gamma_T|}{2}
    $$
    We note that under a uniformly drawn $\alpha$ in $H_m$, we obtain a form of Fano's inequality \cite{Tsybakov} in the probability of misclassification of the best estimator
    $$
    \hat{\alpha} := \argmin_{\alpha \in H_m} \|\hat{f}-f_\alpha\|^2
    $$
    in 
    \begin{align}
    \label{eq:Fanos}
    \MP(\alpha \neq \hat{\alpha}) \geq 1-\frac{\max_{\alpha \neq \beta}\textup{KL}(f_\alpha^{(\otimes m)},f_\beta^{(\otimes m)})+\log(2)}{\log|H_m|}
    \end{align}
    where $f_\alpha^{(\otimes m)}$ is the $m$-fold tensor $f_\alpha^{(\otimes m)} := f_\alpha\otimes f_\alpha\otimes \ldots \otimes f_\alpha$ corresponding to $m$ i.i.d.\ samples of $f_\alpha$, and $\textup{KL}$ denotes the Kullback--Liebler divergence which in this case takes the form
    $$
    \textup{KL}(f_\alpha^{(\otimes m)},f_\beta^{(\otimes m)}) = \int_{M^m} f_\alpha^{(\otimes m)} \log (\frac{f_\alpha^{(\otimes m)}}{f_\beta^{(\otimes m)}})\,d\lambda_M^{(\otimes m)}.
    $$
    It is a matter of straightforward computations to see that $\textup{KL}(f_\alpha^{(\otimes m)},f_\beta^{(\otimes m)}) = m\textup{KL}(f_\alpha,f_\beta)$. By Lemma \ref{lemma:KLDivNonKInvariant} below, 
    $$
    \max_{\alpha\neq \beta}\textup{KL}(f_\alpha,f_\beta) \leq c_1 \eps^2|\Gamma_T|
    $$
    for a constant $c_1 > 0$ if $\eps|\Gamma_T|$ is sufficiently small. We shall soon see that this condition on $\eps |\Gamma_T|$ poses no issues. This means that Inequality \eqref{eq:Fanos} is further bounded as
    $$
        \MP(\alpha \neq \hat{\alpha}) \geq 1-\frac{c_1m\eps^2|\Gamma_T|+\log(2)}{\log|H_m|} \geq 1-c_2\frac{c_1m\eps^2|\Gamma_T|+\log(2)}{|\Gamma_T|}.
    $$
    for some $c_2 > 0$. The last inequality follows from the fact $|H_m| \geq 2^{|\Gamma_T|/8}$ implies that $\log(|H_m|) \geq c_2|\Gamma_T|$. If $\eps^2m$ stays below some sufficiently small positive number as $m \to \infty$, then $\MP(\alpha \neq \hat{\alpha}) \geq c_3 > 0$ for some constant $c_3$. In particular, choosing
    \begin{align}
        \label{eq:probabilityconstraint}
        m\eps^2 \sim 1
    \end{align}
    implies that the probability $\MP(\alpha \neq \hat{\alpha})$ is positive uniformly for elements $f_\alpha \in H_m$.
    
    The proof is almost done. Note that for any $\hat{f} \in H^s(M,Q)$ and any $\alpha$ for which $\alpha \neq \hat{\alpha}$ it is true that
    $$
    \|f_\alpha-f_{\hat{\alpha}}\|_2 \leq \|\hat{f}-f_\alpha\|_2+\|\hat{f}-f_{\hat{\alpha}}\|_2 \leq 2\|\hat{f}-f_{\alpha}\|_2,
    $$
    so 
    $
    \|\hat{f}-f_{{\alpha}}\|^2 \geq \frac{1}{4}\|f_\alpha-f_{\hat{\alpha}}\|^2.
    $
    Because of this
    \begin{align}
    \label{eq:expectationLowerBound}
    \ME\|\hat{f}-f_\alpha\|^2 \geq \ME[\chi_{\{\hat{\alpha} \neq \alpha\}}\cdot\frac{1}{4}\|f_\alpha-f_{\hat{\alpha}}\|^2] \geq \frac{1}{8}\eps^2|\Gamma_T|\MP(\hat{\alpha} \neq \alpha).
    \end{align}
    We know that $\MP(\hat{\alpha} \neq \alpha)$ is positive uniformly in $\alpha$ if $\eps^2 \sim m^{-1}$ by Expression \eqref{eq:probabilityconstraint}, but $\eps$ must also be controlled to ensure that $H_m$ is a subset of $H^s(M,Q)$. Indeed, we note that a sufficient condition for this to be the case is that each element of $H_m$ lies in a ball of small radius about the identity $1 \in H^s(M, Q)$. Moreover,
    $$
    \|f_\alpha- 1\|_{H^s}^2 = \eps^2|\Gamma_T|+\sum_{\pi \in \Gamma_T}\eps^2 \kappa_\pi^s \leq \eps^2|\Gamma_T|+ \eps^2T^s|\Gamma_T| \lesssim \eps^2T^s|\Gamma_T| \sim \eps^2T^{s+d/2},
    $$
    where in the last equality we have utilized the Weyl counting law $|\Gamma_T| \sim T^{d/2}$. This means that under the additional constraint
    \begin{align}
        \label{eq:sobolevconstraint}
        \eps^2T^{s+d/2} \sim 1
    \end{align}
    then $H_m$ is a subset of $H^s(M,Q)$. Combining \eqref{eq:probabilityconstraint} and \eqref{eq:sobolevconstraint}, one sees that if $T$ satisfies
    $$
    T \sim m^{1/(s+d/2)}
    $$
    then $\eps^2 \sim m^{-1}$ and $\eps^2T^{s+d/2} \sim 1$ are possible simultaneously. Furthermore, in this case
    $$
    \eps |\Gamma_T| \sim m^{d/(2s+d)}m^{-1/2} = m^{(d-2s)/(4s+2d)}
    $$
    which tends to zero since $s > \frac{d}{2}$ and cements the fact the assumption $\eps|\Gamma_T| \sim 1$ of Lemma \ref{lemma:KLDivNonKInvariant} is valid. The lower bound in \eqref{eq:expectationLowerBound} therefore further reduces to
    $$
    \ME\|\hat{f}-f_\alpha\|^2 \geq c\eps^2|\Gamma_T|
    $$
    for a constant $c > 0$ independent of $\alpha$. Moreover, $$
    \eps^2|\Gamma_T| \sim m^{-1}m^{d/(2s+d)} = m^{-2s/(2s+d)}
    $$
    follows the desired rate.
    
     Combining the results, we see that
     $$
     \sup_{f\in H^s(M,Q)}\ME\|\hat{f}-f\|^2 \geq \sup_{f_\alpha \in H_m} \ME\|\hat{f}-f_\alpha\|^2 \geq cm^{-2s/(2s+d)}.
     $$
     The bound is evidently independent of the choice of estimator $\hat{f}$ and the result follows.
\end{proof}

%%%%%%%%%% DO NOT TOUCH ABOVE HERE %%%%%%%%%%%%%
Note that Proposition \ref{prop:lowerBoundM} does not show that the estimator in \eqref{eq:mainEstimator} is asymptotically optimal. Indeed, decompounding as framed in Theorem \ref{theorem:densityEstimation} is not a problem of density estimation in $H^s(M)$, but rather in $H^s_K(M,Q) := H^s(M,Q) \cap L_K^2(M)$. As Proposition \ref{prop:lowerBoundMKinvariant} below shows, we obtain a different lower bound for density estimation which depends on the rank of the symmetric space $M$.

\subsection{Lower Bounds for $K$-invariant Density Estimation} To consider $K$-invariant density estimation we must firstly return to considerations of the geometrical structure of $M = G/K$. The Lie algebra $\gg$ of $G$ can be written in terms of its Cartan decomposition $\gg = \hh \oplus \pp$, where $\hh$ is the Lie algebra of $K \subset G$ and where $\pp$ is a Lie subalgebra of $G$ which is isomorphic to the tangent space $T_{p_0}M$ of $M$ at $p_0$. The \textbf{rank} $r$ of $M$ is the dimension of the maximal abelian subalgebra $\mathfrak{a} \subset \pp$. In practical terms, $1 \leq r \leq d$ and $r$ should be interpreted as the maximal dimension of a totally geodesic flat submanifold of $M$. For example, $S^d$ has rank $r = 1$ for all $d$, since we can at most embed a great circle (of dimension $1$) as a totally geodesic flat submanifold around any point. Likewise, $\real^d$ is trivially of rank $d$. We reiterate once again that a formal treatment of these classical facts is found in \cite{Helgason}.

For lower bounds on density estimation in $L_K^2(M)$, we will return to considering sets of the form
$$
F_T :=\{\pi \in \hat{G}_K:\kappa_\pi \leq T\},
$$
and it is because of this that the rank $r$ has an influence on density estimation. Indeed, in $F_T$ there is only a single element $\pi$ for each Casimir element, as opposed to the $d_\pi$ basis elements that are included for each Casimir element in the set $\Gamma_T$ constructed for Proposition \ref{prop:lowerBoundM}. Hence, the Weyl counting law can not be applied directly as it could for $\Gamma_T$. Instead, the growth rate of this set will be
$$
|F_T| \sim T^{r/2}.
$$
Let us give a sketch of the main steps needed to prove why this is true: By the classical statement of the Cartan--Helgason-theorem, see \cite[Chapter 5]{Helgason3}, the spherical representations $\pi \in \hat{G}_K$ can be parametrized by a subset of the dominant weights of $G$ that restrict to linear functionals $\gamma_\pi: \mathfrak{a} \to \real$ under certain invariance conditions, where again $\mathfrak{a} \subset \pp$ is a maximal Abelian subspace. We aptly say that the resulting restrictions $W := \{\gamma_\pi: \mathfrak{a} \to \real\}$ are the \textbf{restricted weights} of $M$. It is a classical result, see e.g.\ \cite[Corollary 2.5.1]{Applebaum}, that a dominant weight $\gamma$ corresponding to $\pi \in \hat{G}$ satisfies $|\gamma|^2 \sim \kappa_\pi$. In other words, counting elements in $F_T$ amounts to counting the number of restricted weights that are contained in a ball of radius $T^{1/2}$ in the algebraic dual $\mathfrak{a}^*\cong \real^r$ of $\mathfrak{a}$. Due also to the Cartan--Helgason-theorem, the set of restricted weights $W = L \cap \mathcal{C}$ is given as the intersection by an $r$-dimensional lattice $L \subset \mathfrak{a}^*$ and a closed cone\footnote{In actuality, $\mathcal{C}$ is a positive Weyl-Chamber of the root system in $\mathfrak{a}^*$, but this would require a paragraph on its own to define and we thus leave it as generically defined as possible.} $\mathcal{C}$ with non-empty interior and full rank $r$. In other words, the number of restricted weights contained in a ball of radius $T^{1/2}$ has the same leading order growth as lattice point counting in $\real^r$ since the intersection by $\mathcal{C}$ is immaterial for the leading order growth of $W$. We conclude that
$$
|F_T| \sim |\{\gamma \in W: |\gamma|\leq T^{1/2}\}| \sim |\{\gamma \in L \subset  \mathfrak{a}^*: |\gamma| \leq T^{1/2}\}|  \sim T^{r/2}.
$$
With this, we are ready to state and prove Proposition \ref{prop:lowerBoundMKinvariant}.
%%%%%%%%%%%%%%%%%%%%%%% DO NOT TOUCH BEYOND HERE %%%%%%%%%%%%%%%%%%
\begin{proposition}
\label{prop:lowerBoundMKinvariant}
    Let $\Theta^K_m$ be the set of estimators $\hat{f}_m$ of some $f \in H_K^s(M,Q)$ given $m$ i.i.d.\ observations of $f$. If $s > d/2$, then
    $$
    \inf_{\hat{f} \in \Theta_m^K} \sup_{f\in H^s(M,Q)} \ME\|\hat{f}-f\|^2 \geq Cm^{-2s/(2s+r)}
    $$
    for some $C >0$, where $r$ is the rank of $M$.
\end{proposition}
\begin{proof}
    Let $F_T := \{\phi_\pi: \kappa_\pi \leq T\}$ and $\eps > 0$. Consider the set consisting of functions $f_\alpha = 1+\eps\sum_{\pi \in F_T}\alpha_\pi \sqrt{d_\pi}\phi_\pi$ for $\alpha = (\alpha_\pi) \in \{-1,1\}^{F_T}$ and construct a subset $H_m$ wherein distinct indices $\alpha$ and $\beta$ of $f_\alpha,f_\beta \in H_m$ have Hamming distance greater than $|F_T|/8$ while $|H_m| \geq 2^{|F_T|/8}$. Then, for any two $f_\alpha \neq f_\beta$ in $H_m$ it follows that
    $$
    \|f_\alpha-f_\beta\|^2 = \eps^2 \sum_{\pi \in \Gamma_T}|\alpha_\pi-\beta_\pi|^2 \geq \frac{\eps^2|F_T|}{2}.
    $$
    By Lemma \ref{lemma:KLDivKInvariant}, the Kullback--Leibler divergences satisfy 
    $$
    \textup{KL}(f_\alpha,f_\beta) \lesssim \eps^2|F_T|
    $$
    under the assumption that $\eps T^{(r+d)/4}$ remains sufficiently small. We obtain, identically to the proof of Proposition \ref{prop:lowerBoundM} that 
    $$
    \ME\|\hat{f}-f_\alpha\|^2 \geq \frac{1}{8}\eps^2|F_T|\MP(\hat{\alpha} \neq \alpha)
    $$
    for $\hat{\alpha} := \argmin_{\alpha \in H_m} \|\hat{f}-f_\alpha\|^2$, and 
    $$
    \MP(\hat{\alpha} \neq \alpha) \geq 1-c_2\frac{c_1m\eps^2|F_T|+\log(2)}{|F_T|}.
    $$
    for constants $c_1,c_2 > 0$, by identical argumentation to that of the proof of Proposition \ref{prop:lowerBoundM}. Hence, the condition that $m\eps^2 \sim 1$ for a small bounding constant is sufficient to conclude that
    $$
    \ME\|\hat{f}-f_\alpha\|^2 \geq C\eps^2|F_T|
    $$
    for a constant $C$ independent of $\alpha$. We see furthermore that
    $$
    \|f_\alpha-1\|_{H^s(M)}^2 = \eps^2|F_T|+\eps^2\sum_{\pi \in F_T}\kappa_\pi^s \leq \eps^2|F_T|+\eps^2T^s|F_T|,
    $$
    meaning that $H_m \subset H^s_K(M,Q)$ if $\eps^2T^{s+r/2}$ remains bounded. The conditions $\eps^2m \sim 1$ and $\eps^2T^{s+r/2} \sim 1$ can be made possible simultaneously by choosing $\eps^2 \sim m^{-1}$ and $T \sim m^{1/(s+r/2)}$. Such a choice does not violate the condition that $\eps T^{(r+d)/4}$ remains bounded in Lemma \ref{lemma:KLDivKInvariant}. Indeed, for these choices of $\eps$ and $T$ we obtain that
    $$
    \eps T^{(r+d)/4} \sim m^{-1/2}m^{(r+d)/(4s+2r)} = m^{(r+d-(2s+r))/(4s+2r)} = m^{(d-2s)/(4s+2r)},
    $$
    which tends to zero since $s > \frac{d}{2}$. Consequently,
    $$
    \eps^2|F_T| \sim m^{-1}T^{r/2}\sim m^{-1}m^{r/(2s+r)} = m^{-2s/(2s+r)} 
    $$
    and in particular
    $$
    \sup_{f \in H^s_K(M,Q)}\ME\|\hat{f}-f\|^2 \geq \sup_{f_\alpha \in H_m}\ME\|\hat{f}-f_\alpha\|^2 \geq Cm^{-2s/(2s+r)}
    $$
    for a constant $C$ which does not depend on the estimator $\hat{f}$. The result follows.
\end{proof}
When comparing $m^{-2s/(2s+r)}$ and $m^{-2s/(2s+d)}$ one should firstly note the obvious fact that the relative difference between the rates is greater for low values of $r$. However, we also see that as $s$ increases, the difference between the rates is less significant. In other words, the smoother we assume that the densities are, the more negligible is the difference between the lower bounds in $H_K^s(M,Q)$ and those in $H^s(M,Q)$.

Additionally, the fact that the decompounding estimator in Theorem \ref{theorem:densityEstimation} does not meet the rate provided by Proposition \ref{prop:lowerBoundMKinvariant} is not that surprising when considering the fact that it does not utilize any information of the maximal abelian subspace $\mathfrak{a} \subset\pp$. However, we should also note that strictly speaking, Proposition \ref{prop:lowerBoundMKinvariant} does not exclude the possibility that decompounding is an estimation problem that can not reach the rates of the lower bound. Indeed, Proposition \ref{prop:lowerBoundMKinvariant} provides a lower bound, but it may not necessarily be the greatest lower bound. 
\subsection{Proofs of Lemmas}
Before proving Lemmas \ref{lemma:KLDivNonKInvariant} and \ref{lemma:KLDivKInvariant}, we recall that the Weyl--Law can be stated in terms of the eigenfunctions $(e_k)_{k = 1}^\infty$ alternatively, in that \cite{Helgason}
$$
\sum_{\lambda_k \leq T}|e_k(p)|^2 \lesssim T^{d/2}
$$
uniformly in $p \in M$.
\begin{lemma}
\label{lemma:KLDivNonKInvariant}
    Let $H_m$ and $\Gamma_T$ be defined as in the proof of Proposition \ref{prop:lowerBoundM}.  If $\eps T^{d/2} \sim 1$ for a sufficiently small bounding constant, then $f_\alpha \geq \frac{1}{2}$ for all $f_\alpha \in H_m$. In that case,
    $$
    \textup{KL}(f_\alpha,f_\beta) \lesssim \eps^2|\Gamma_T|.
    $$
\end{lemma}
\begin{proof}
    For $f_\alpha$ to be greater than $\frac{1}{2}$, it is sufficient that $|f_\alpha-1| \leq \frac{1}{2}$. By applying Cauchy--Schwarz we obtain
    $$
    |f_\alpha-1| \leq \eps\sum_{\lambda_k \leq T_m} |e_k| \leq \eps \sqrt{|\Gamma_T|}\bigg(\sum_{\lambda_K\leq T_m}|e_k|^2\bigg)^{1/2}\sim \eps |\Gamma_T| \sim \eps T^{d/2}.
    $$
    This proves the first claim. For the second claim, we let $B$ be a bounded real function for which $\log(1+x) = x+B(x)x^2$. Note in particular that $C:= \sup_{p \in M}B(\frac{|f_\alpha-f_\beta|^2}{f_b^2}) < +\infty$ if $\eps T^{d/2}$ is chosen sufficiently small. Then
    \begin{align*}
        \textup{KL}(f_\alpha,f_\beta) &= \int_M f_\alpha \log\Big(1+\frac{f_\alpha-f_\beta}{f_\beta}\Big)\,d\lambda_M\\
        & = \int_Mf_\alpha(\frac{f_\alpha-f_\beta}{f_\beta}+B(\xi)\Big(\frac{f_\alpha-f_\beta}{f_\beta}\Big)^2)\,d\lambda_M
    \end{align*}
    for some $\xi$ close to $\frac{f_\alpha-f_\beta}{f_\beta}$. By rewriting $f_\alpha = f_\alpha-f_\beta+f_\beta$ this further expands to
    \begin{align*}
        &\int_M (f_\alpha-f_\beta)(\frac{f_\alpha-f_\beta}{f_\beta}+B(\xi)\Big(\frac{f_\alpha-f_\beta}{f_\beta}\Big)^2)\,d\lambda_M + \int_Mf_\alpha-f_\beta + B(\xi)\Big(\frac{(f_\alpha-f_\beta)^2}{f_\beta}\Big)\,d\lambda_M\\
        &  = \int_M \frac{1}{f_\beta}(f_\alpha-f_\beta)^2\,d\lambda_M+\int_M \Big(\frac{(f_\alpha-f_\beta)^2}{f_\beta}+\frac{(f_\alpha-f_\beta)^3}{f_\beta^2})\Big)B(\xi)\,d\lambda_M.
    \end{align*}
    We apply a modulus sign to the entire expression, which gives that
    $$
    \textup{KL}(f_\alpha,f_\beta) = |\textup{KL}(f_\alpha,f_\beta)| \leq \frac{1}{2}\|f_\alpha-f_\beta\|^2 +C\int_M \frac{1}{2}|f_\alpha-f_\beta|^2+\frac{1}{4}|f_\alpha-f_\beta|^3\,d\lambda_M \lesssim \|f_\alpha-f_\beta\|^2 \sim \eps^2|\Gamma_T|
    $$
    which was to be proven.
\end{proof}
\begin{lemma}
\label{lemma:KLDivKInvariant}
    Let $H_m$ and $F_T$ be defined as in the proof of Proposition \ref{prop:lowerBoundMKinvariant}. If $\eps T^{(r+d)/4} \sim 1$ for a sufficiently small bounding constant, then $f_\alpha \geq \frac{1}{2}$ for all $f_\alpha \in H_m$. In that case,
    $$
    \textup{KL}(f_\alpha,f_\beta) \lesssim \eps^2|F_T|.
    $$
\end{lemma}
\begin{proof}
    In this case
    $$
    |f_\alpha -1| \leq \eps \sum_{\pi \in \Gamma_T}\sqrt{d_\pi}|\phi_\pi| \leq \eps \bigg(\sum_{\pi \in \Gamma_T}d_\pi\bigg)^{1/2}\bigg(\sum_{\pi \in \Gamma_T}1\bigg)^{1/2} \sim \eps T^{d/4}T^{r/4}
    $$
    since $|\phi_\pi(p)|\leq 1$ for all $p \in M$. It is clear that if $\eps T^{(r+d)/4}$ is chosen sufficiently small, then $f_\alpha \geq \frac{1}{2}$ for all $p \in M$ and all $f_\alpha \in H_m$. The rest of the proof is identical to that of Lemma \ref{lemma:KLDivNonKInvariant}.
\end{proof}

\section{Effect of Noisy Observations}
\label{Sec:FurtherWork}
So far, we have not discussed the influence that noise can have on the observations. A natural assumption is that the observations $p_1,p_2,\ldots, p_m$ are perturbed into
$$
z_k = \eps_k\cdot p_k
$$
for i.i.d.\ Riemannian brownian noise variables $\eps_k$, i.e. pure diffusion processes with diffusion term $\tau\Delta$ for some $\tau > 0$. This means by \eqref{eq:LevyKhinchinM} that $\ME[\phi_\pi(\eps_k)] = e^{-\frac{\tau^2\kappa_\pi}{2}}$. The method of taking empirical averages over $z_k$ instead of $p_k$ in Equation \eqref{eq:empiricalAverage} therefore tends towards $e^{-\frac{\tau^2\kappa_\pi}{2}}\nu_t(\phi_\pi)$, and in particular
$$
\langle f_X,\phi_\pi\rangle = \frac{1}{t\Lambda}\log(\nu_Z(\phi_\pi))+(1+\frac{\tau^2\kappa_\pi}{2t\Lambda}),
$$
where $\nu_Z$ is the distribution of $\eps_k\cdot p_k$. If the intensity of the noise is known \textit{a priori}, this poses no issue in decompounding since the corresponding term simply needs to be included in an altering of the estimator in Proposition \ref{prop:L1Estimator} into
\begin{align}
\label{eq:estimatorWithNoise}
    c^\eps_m(\pi) := \begin{cases}
        \frac{1}{t\Lambda}\log(\nu_m)+(1+\frac{s^2\kappa_\pi}{2t\Lambda}) & \textup{if } \nu_m \geq \frac{\delta}{m}\\
        0 & \textup{otherwise,}
    \end{cases}
\end{align}
from which the quality of the estimator remains unchanged. Indeed, the arguments for the convergence rates in Proposition \ref{prop:L1Estimator} are not altered since $\frac{(\phi_\pi(\eps_k\cdot p_k)+\overline{\phi_\pi(\eps_k\cdot p_k)})}{2}$ remains to be a real-valued random variable bounded on $[-1,1]$. However, a more realistic assumption is either that $\tau$ is entirely unknown, or that it is jointly approximated by some estimator $\hat{\tau}_m$ of $\tau$. Assuming that we are in the event where $\nu_m \geq \frac{\delta}{m}$, the error then takes the form
$$
|c_m^\eps(\pi)-\langle f_X,\phi\rangle| = |\frac{1}{t\Lambda}\big(\log(\nu_m)-\log(\nu_Z(\phi_\pi))\big)+\kappa_\pi\frac{\hat{\tau}_m^2-\tau^2}{2t\Lambda}|,
$$
which is troublesome since the Casimir elements $\kappa_\pi$ get arbitrarily large. Finding estimators $\hat{\tau}_m$ that utilize the perturbed observations would therefore be a worthwhile pursuit, especially if convergence can be ensured that are fast enough for the contribution of $\kappa_\pi(\hat{\tau}_m^2-\tau^2)$ to be negligible in the estimator of Equation \eqref{eq:estimatorWithNoise}.

Another approach is to apply a convolution mask to the series expansion. One such example would be to adjust the density estimator into
$$
f_X^{(m)} := \sum_{\pi \in \Gamma_m} d_\pi e^{-C\kappa_\pi}c^\eps_{m}(\pi)\phi_\pi
$$
for a positive constant $C > 0$, but this evidently sacrifices the accuracy of the estimator. For application purposes, it would be useful to perform simulated experiments that compare how these altercations fare, and how problematic disturbances by noise truly are.

%meaning explicitly that for any two vector fields $\xi,\eta$ on the Lie algebra $\mathfrak{g}$ of $G$ the metric $\rho$ is invariant under the adjoint representation $\textup{\textbf{Ad}}$ for all $x \in G$:
%$$
%\rho(\xi,\eta) = \rho(\textup{\textbf{Ad}}(x)\xi,\textup{\textbf{Ad}}%(x)\eta)
%$$
%for all $\xi,\eta \in \mathfrak{g}$ and all $x \in G$.

\section{Conclusion}
\label{sec:Conclusion}
Estimators for decompounding on compact symmetric spaces have been constructed and they converge in mean squared error at a rate of $m^{-2s/(2s+d)}$. We have shown that this is a lower bound on density estimation in compact symmetric spaces, but that the bound is not necessarily tight for the decompounding problem where we furthermore assume $K$-invariance on the density. As such, it was furthermore shown that under this assumption a lower bound is found in $m^{-2s/(2s+r)}$ instead. This means that the decompounding method is asymptotically optimal for compact symmetric spaces with $d =r$, i.e.\ the flat toruses $\mathbb{T}^d$, and that the asymptotic rate of convergence strays further from the proven lower bound for lower ranks of $M$. For large $s$, i.e.\ as the smoothness of the densities increase, the difference between the lower bound and attained rate of the decompounding method is negligible. 

Perturbations on the observations by Gaussian noise lead to worse convergence rates unless information of the intensity of the noise is known. Further investigations are needed on whether estimations of the noise intensity can at least partly recover convergence rates. In addition, with slight alterations to the estimator in Proposition \ref{prop:L1Estimator} there is no need to assume that the distribution of the steps are inverse invariant. The largest gap to consider for further work lies in whether the decompounding estimator can be improved to reach the general lower bound given by Proposition \ref{prop:lowerBoundMKinvariant}.

The deconvolution problem appears naturally in statistical problems in engineering applications and the natural sciences, and can now be formulated and solved on a broader range of commonly appearing manifolds that are not necessarily Lie groups.

% or


%\appendices

%\section{}

%\section*{Acknowledgment}


\ifCLASSOPTIONcaptionsoff
  \newpage
\fi

\end{document}